\documentclass[11pt,a4paper]{amsart}
\usepackage{mathrsfs}
\usepackage{latexsym}
\usepackage{amsthm}
\usepackage{amsfonts}
\usepackage[usenames]{color}
\usepackage{amssymb}
\usepackage{graphicx}
\usepackage{amsmath}
\usepackage{amsfonts}
\usepackage{amsthm}
\usepackage{mathrsfs}
\usepackage{dsfont}
\usepackage{cases}
\usepackage{enumerate}
\usepackage{indentfirst}
\usepackage{accents}
\usepackage{geometry}
\newtheorem{definition}{Definition}[section]
\newtheorem{theorem}{Theorem}[section]

\newtheorem{proposition}{Proposition}[section]
\newtheorem{remark}{Remark} [section]

\begin{document}
\title{An inverse problem in optimal transport on closed Riemannian manifolds}

  \author[J. Zhai]{Jian Zhai}
\address{J. Zhai: School of Mathematical Sciences, Fudan University, 220 Handan Road, Shanghai 200433, China; Center for Applied Mathematics, Fudan University, 220 Handan Road, Shanghai 200433, China; 
  (\tt{jianzhai@fudan.edu.cn}).}
\thanks{J. Zhai is supported by National Key Research and Development Programs of China (No. 2023YFA1009103), NSFC(No. 12471396), Science and Technology Commission of Shanghai Municipality (23JC1400501)}

 \author[S. Zhang]{Kelvin Shuangjian Zhang}
\address{K. S. Zhang: School of Mathematical Sciences, Fudan University, 220 Handan Road, Shanghai 200433, China; Center for Applied Mathematics, Fudan University, 220 Handan Road, Shanghai 200433, China; 
  (\tt{ksjzhang@fudan.edu.cn}).}

\begin{abstract}
We consider the problem of recovering the Riemannian metric on a compact closed manifold from the optimal transport maps when the underlying cost function is the squared Riemann distance. We show that the metric can be uniquely determined up to a multiplicative constant.
\end{abstract}
\maketitle
\section{Introduction}

The optimal transport problem is to find the transport plans between two probability measures that minimize the overall transport cost for a given cost function $c$. The inverse optimal transport problem we consider in this paper aims to recover the underlying cost $c$ that results in the optimal transport plans. Inverse OT problems have drawn a lot attention recently, see for example~\cite{stuart2020inverse,paty2020regularized,ma2020learning,gonzalez2024nonlinear,bao2025wellposednessefficientalgorithmsinverse, wang2023self}. 

Optimal transport was first studied by Monge \cite{monge1781memoire} in 1781, while Kantorovich \cite{kantorovich1942translocation} in 1942 proposed a relaxed version of the Monge problem and provided its dual formulation. The modern theory of optimal transport was initiated by Brenier \cite{brenier1987decomposition,brenier1991polar} in 1991, where he showed the existence, uniqueness, and characterization of optimal transport maps when the cost function is the squared Euclidean distance. 
%If talk about the conditions on $\mu$ and $\nu$, state here about Brenier's condition and McCann's relaxation in \cite{Duke}.
%The regularity of optimal transport maps in Euclidean spaces was studied in \cite{caffarelli1992regularity,caffarelli1996boundary,delanoe2006regularity,figalli2010partial,liu2017regularity}. 
The study of the optimal transport problem on Riemannian manifolds can be traced back to the work of McCann \cite{mccann2001polar}, where he proved the existence and uniqueness of optimal transport maps when the cost function is the squared Riemannian distance function. We refer to \cite{villani2008optimal,santambrogio2015optimal} for more details on optimal transport.

Vast applications of optimal transportation can be found in other fields of mathematics, economics, finance, and biology. In biology, researchers seek to understand the ancestors and descendants of cells during the cellular differentiation stage of embryonic development. One way to achieve this goal is, knowing the empirical distributions of cells from experiments at, say, two given time points, to find the optimal transportation plan that gives the joint distribution of these marginals under the quadratic cost~\cite{schiebinger2019optimal}. In the literature, researchers assume that cells embedded in a high-dimensional Euclidean space have a natural Euclidean quadratic cost by default. 

However, if the cost function is different, the optimal transport plan may also change. Thus, we are curious about the inverse problem of recovering the transportation cost function defined on the product space, given any two marginals and knowing their optimal transportation plans, i.e., the joint distributions, under this cost. 

In this paper, we study an inverse optimal transport problem on a closed Riemannian manifold, and the cost function is the squared Riemannian distance function.
Assume $(M,g)$ is a compact connected Riemannian manifold without boundary. Denote $d(x,y)$ be the distance between the two points $x$ and $y$, i.e., the length of the shortest path connecting $x$ and $y$.  Denote $\mathrm{d}V_g$ to be the volume form of $(M,g)$, where in local coordinates $\mathrm{d}V_g(x)=\sqrt{\det(g_{ij}(x))}\mathrm{d}x^1\wedge\cdots\wedge\mathrm{d}x^n$. We define the exponential map $\exp_x:T^*_xM\rightarrow M$ as
\[
\exp_x(v)=\gamma_{x,v}(1)
\]
where $\gamma_{x,v}$ is a geodesic such that $\gamma_{x,v}(0)=x$, $\dot{\gamma}_{x,v}(0)=v^\sharp$. Here $v^\sharp=g^{-1}v\in T_xM$.
%{\red Here introduce the exponential map. Also, in order to make sense of the volume form, one might need some anzats on smoothness of the metric. }

Denote the cost function by $c(x,y)=\frac{1}{2}d^2(x,y)$. %Let $P(M)$ denote the space of probability measures on $M$.
Let $P_{ac}(M)$ denote the space of probability measures on $M$ that are absolutely continuous with respect to $\mathrm{d}V_g$. Given $\mu,\nu\in P_{ac}(M)$,
%\marginpar{SJ: here one can define the Monge problem for P(M) but concentrate on the problem restricted on $P_{ac}(M)$.}
consider the Monge minimization problem
\begin{equation}
\inf_{\Phi_\#\mu=\nu}\int _M c(x,\Phi(x))\mathrm{d}\mu(x).
\end{equation}
Here the infimum is taken over all measurable maps $\Phi:M\rightarrow M$ satisfying $\Phi_\#\mu=\nu$, where $(\Phi_\#\mu)[A]=\mu[\Phi^{-1}(A)]$ for any measurable set $A$. For the solution of the above Monge problem, McCann \cite{mccann2001polar} proved:
\begin{proposition}\label{OT_existence-and-uniqueness}
%Let $\mathrm{d}\mu(x)=f(x)\mathrm{d}V_g(x)$ and  $\mathrm{d}\nu(y)=h(y)\mathrm{d}V_g(y)$. 
Let $\mu(x),\nu(x)\in P_{ac}(M)$.
Then there exists a unique solution $\Phi^\star$ (uniqueness up to $\mu$-a.e.) such that
\[
\int _M c(x,\Phi^\star(x))\mathrm{d}\mu(x)=\inf_{\Phi_\#\mu=\nu}\int _M c(x,\Phi(x))\mathrm{d}\mu(x).
\]
Furthermore, $\Phi^\star$ is given by $\Phi^\star(x)=\exp_x(\nabla\psi(x))$ for some $c$-convex function $\psi:M\rightarrow \mathbb{R}$.
\end{proposition}
%\marginpar{SJ: McCann\cite{mccann2001polar} only require $\mu$ is absolutely continuous, while $\nu$ need merely Borel. We might need to add a note on this.}
%{\color{red}\textbf{[Uniqueness and smoothness of $\psi$ near $\psi\equiv c$. i.e., $f$ is close to $h$.]}}
\begin{remark}
    In McCann's work \cite{mccann2001polar}, one actually can only assume $\mu$ is absolutely continuous with respect to $\mathrm{d}V_g$ and $\nu$ needs merely Borel.
\end{remark}

The $c$-convexity is defined as follows
\begin{definition}[$c$-convexity]
A function $\psi:M\rightarrow\mathbb{R}\cup\{+\infty\}$ is $c$-convex if for any $x$
\[
\psi(x)=\sup_{y\in M}[\psi^c(y)-c(x,y)],
\]
where
\[
\psi^c(y)=\inf_{x\in M}[\psi(x)+c(x,y)].
\]
\end{definition}

We consider the following inverse problem: assuming that, under the squared geodesic distance cost function, the optimal transport map between any two probability measures $\mu,\nu \in P_{ac}(M)$ is known, we aim to recover the Riemannian metric $g$. The problem of recovering a general cost function from optimal transport data has been analyzed in \cite{gonzalez2024nonlinear}.\\

For any two measures $\mu,\nu \in P_{ac}(M)$, we may write $\mathrm{d}\mu(x)=f(x)\mathrm{d}V_g(x)$ and  $\mathrm{d}\nu(y)=h(y)\mathrm{d}V_g(y)$ where $f, h \ge 0$ and 
$\int_Mf\mathrm{d}V_g=\int_Mh\mathrm{d}V_g=1$. We also assume throughout the paper that $h>0$ on $M$.
 Denote by $\Phi$ the optimal transport map described in Proposition \ref{OT_existence-and-uniqueness}. 
 It was shown in \cite{cordero2001riemannian} that the solution $\Phi$ is differentiable $\mu$-almost everywhere and satisfies the Jacobian equation
\begin{equation}\label{Jacobian}
|\det(\mathrm{d}_x\Phi)|=\frac{f(x)}{h(\Phi(x))}.
\end{equation}
We refer to \cite[Chapter 11]{villani2008optimal} for more details and here the Jacobi determinant $\det(\mathrm{d}_x\Phi)$ is associated with $T$ as
\begin{equation}\label{algebraicjacobi}
|\det(\mathrm{d}_x\Phi)|=\lim_{r\rightarrow 0}\frac{\mathrm{vol}_g(\Phi(B_r(x))}{\mathrm{vol}_g(B_r(x))}.
\end{equation}
%\marginpar{SJ: here $T$ and its multi-valued extension can be exchangable?}
We can also define $\det(\mathrm{d}_x\Phi)$ in an algebraic way \cite[page 364]{villani2008optimal}.
For this, denote $\gamma$ a geodesic from $x$ to $ \Phi(x)$ such that $\gamma(0)=x$. Choose an orthonormal basis of $T_xM$, denoted by $\{e_1,e_2,\cdots, e_n\}$ where $e_1=\frac{\dot{\gamma}(0)}{|\dot{\gamma}(0)|}$. We then parallel transport $e_j$ from $x=\gamma(0)$ to any point $\gamma(t)$ on the geodesic to obtain $e_j(t)$, $j=1,2,\cdots, n$. Then we get an orthonormal basis $\{e_1(t),e_2(t),\cdots, e_n(t)\}$ of $T_{\gamma(t)}M$ for any $t$. Consider $\Phi$ as a map $\Phi:M\rightarrow M$, and then $\mathrm{d}_x\Phi$ is a map
\[
 \mathrm{d}_x\Phi:T_xM\rightarrow T_{\Phi(x)}M.
\]
%\marginpar{SJ: I suggest that change the letter for the notation of $T$ for $T(x)$ to another letter. As $T$ also used for tangent spaces.}
Under the above orthonormal bases for $T_xM$ and $T_{\Phi(x)}M$, one can write $\mathrm{d}_x\Phi$ as a matrix. Then $\det(\mathrm{d}_x\Phi)$ is simply defined as the determinant of the matrix. One can show that this algebraic definition and the geometric definition \eqref{algebraicjacobi} are equivalent. See also \cite[Claim 4.5]{cordero2001riemannian}.\\

Since
\[
\Phi(x)=\exp_x(\nabla\psi(x)),
\]
we have $\nabla\psi(x)=\exp_x^{-1}(\Phi(x))$.
Then
\begin{equation*}
\nabla_xd(x,\Phi(x))=\dot{\gamma}(0)=\frac{\nabla\psi(x)}{|\nabla\psi(x)|},\quad d(x,\Phi(x))=|\nabla\psi(x)|.
\end{equation*}
Here $\nabla_x d(x,y)$ means the gradient of $d(x,y)$ as a function of the first variable (similar for $\nabla_x c(x,y)$, see below).
So we end up with the equation
\begin{equation}\label{eqpsi}
\nabla\psi(x)+\nabla_xc(x,\Phi(x))=0.
\end{equation}
Differentiating the above equation with respect to $x$, we obtain
\[
\nabla^2\psi(x)+\nabla^2_xc(x,\Phi(x))+\nabla_y\nabla_xc(x,\Phi(x))\mathrm{d}_x\Phi(x)=0.
\]
Here $\nabla^2_xc(x,y)$ means $\nabla^2 c(\cdot,y)$ and $\nabla_y\nabla_xc(x,y)$ means first differentiating $c(x,y)$ in $x$ with $y$ fixed and then differentiating in $y$ with $x$ fixed.
Using the Jacobian equation \eqref{Jacobian}, we have the Monge-Amp\`ere type equation
\begin{equation}\label{MAeq}
\det\left[\nabla^2\psi(x)+\nabla^2_xc(x,\Phi(x))\right]=|\det(\nabla_y\nabla_xc(x,\Phi(x)))|\frac{f(x)}{h(\Phi(x))}.
\end{equation}
First, we observe that, whatever metric $g$ is, as long as the densities $f$ and $h$ are the same, any constant function $\psi\equiv C$ is a {solution} to \eqref{MAeq} given that $c = \frac{1}{2}d^2$. Indeed, if $\psi\equiv C$, we have $\nabla^2\psi=0$ and $\Phi(x)=x$.  Using the facts
\begin{equation}\label{distancetometric}
\nabla^2_xc(x,y)\vert_{y=x}=g(x),\quad \nabla^2_{x,y}c(x,y)\vert_{y=x} =-g(x),
\end{equation}
one can immediately see that $\psi\equiv C$ is the solution to the Monge-Amp\`ere type equation. Note that under orthonomal coordinates at $x$, $\nabla^2_xc(x,y)\vert_{y=x}$ and $\nabla^2_{x,y}c(x,y)\vert_{y=x}$ are simply $I_n$ and $-I_n.$\\
%the last sentence is not necessary.

%Note that $\nabla\psi$ is uniquely determined by the unique optimal transport map $T$ by \eqref{eqpsi}. %
Now the inverse problem under investigation can be formulated based on the above Monge-Amp\`ere type equation: to recover $g$ assume knowing $T$ solving \eqref{MAeq} for any $f\mathrm{d}V_g$ and $h\mathrm{d}V_g$ such that $\int_Mf\mathrm{d}V_g=\int_Mh\mathrm{d}V_g=1$ and $h> 0$. We mention that an inverse source problem for a Monge-Amp\`ere equation has been recently studied in \cite{LiLin2025} when we were preparing for the manuscript.

\begin{remark}
We note here that we can only specify $\mu$ and $\nu$, not really $f$ and $h$, for the metric $g$ is not known a priori. 
\end{remark}

\begin{remark}
The optimal map should remain the same if one multiply $g$ by a positive constant $c_0$.
\end{remark}

The main result of our paper is:
\begin{theorem}\label{mainthm}
Assume for any $\mu(x)=f\mathrm{d}V_g,\nu(x)=h\mathrm{d}V_g\in P_{ac}(M)$ with $f, h\in C^\infty(M)$ and $h>0$, the solution $\Phi(x)$ to the equation \eqref{MAeq} is known.
Then the metric $g$ can be uniquely determined up to a multiplicative constant.
\end{theorem} 
%\marginpar{SJ: I think one should include all the assumptions and settings in the main theorem so that the readers won't need to read through the intro to find all the assumptions. Do we define a space for the assumption $h>0$ on $M$, so that only need the information of transport maps between two absolutely continuous probability densities with the second density being strictly positive on $M$?}

\section{Linearization}
The Monge-Amp\`ere type equation \eqref{MAeq} is a fully nonlinear elliptic equation, we first linearize it at $f\equiv h$, for which constant function $\psi\equiv C$ is a solution.\\

We first examine $\nabla^2_xc(x,y)$ and $\nabla_y\nabla_xc(x,y)$. We will follow the lines in \cite[Chapter 14, Third Appendix]{villani2008optimal} closely.
Assume $\gamma$ is a geodesic from $x$ to $y$ such that $\gamma(0)=x,\gamma(1)=y$. Denote $v=\exp_x^{-1}(y)$, that is $y=\exp_xv$. Denote
\[
F_t(x,v)=\exp_x(tv).
\]
Using the relation
\begin{equation}\label{Ftexp}
F_1(x,\exp_x^{-1}(y))=y,
\end{equation}
we have
\[
(\partial_xF_1+\partial_vF_1\nabla_x\exp_x^{-1}y)=0
\]
by taking derivatives with respect to $x$ of \eqref{Ftexp}.

Note that $J_0^1(t)=\partial_xF_t$ and $J_1^0(t)= \partial_vF_t$, as variations of geodesics, are Jacobi matrices. More precisely, $J_0^1(t)$ and $J_1^0(t)$ satisfy the Jacobi equations
\begin{equation}\label{JacobiEq}
\ddot{J}_\sigma^{1-\sigma}(t)+R(t)J_\sigma^{1-\sigma}(t)=0,\quad \sigma=0,1,
\end{equation}
 such that
\[
J_0^1(0)=I_n,\quad \dot{J}_0^1(0)=0,\quad J_1^0(0)=0,\quad \dot{J}_1^0(0)=I_n,
\]
under the parallel-transported orthonormal basis $\{e_1(t),e_2(t),\cdots, e_n(t)\}$.
Here $R$ is a symmetric matrix where
\[
R_i^j(t)=\langle R(\dot{\gamma}(t),e_i(t))\cdot\dot{\gamma}(t),e_j(t)\rangle.
\]
where $R$ is the Riemann curvature tensor.
Then
\[
J_0^1(1)-J_1^0(1)H=0,
\]
where
\[
H=\nabla^2_xc(x,y).
\]
Taking derivatives in $y$ of \eqref{Ftexp}, we obtain
\[
\partial_vF_1\nabla_y\exp_x^{-1}(y)=I_n,
\]
which can be written as
\[
-J_1^0(1)\nabla_y\nabla_x c(x,y)=I_n
\]
using the relation
\begin{equation*}
\nabla_x c(x,y)+\exp_x^{-1}(y)=0.
\end{equation*}
~\\

Although $f$ and $h$ are not fully known, the ratio $\frac{f}{h}=\frac{\mathrm{d}\mu}{\mathrm{d}\nu}$ is determined by $\mu$ and $\nu$ without knowing $\mathrm{d}V_g$. Now take $\mu=\nu+\varepsilon\mu_1$ with $\mathrm{d}\mu_1= f_1\mathrm{d}V_g$ and $\int_M f_1\mathrm{d}V_g=0$. Then $\frac{f_1}{h}=\frac{\mathrm{d}\mu-\mathrm{d}\nu}{\mathrm{d}\nu}$. Recall that we have chosen $h$ to be strictly positive. So for a fixed $f_1$, when $\varepsilon$ is sufficiently small $f$ is also strictly positive. By the discussion after \cite[Theorem 10.28]{villani2008optimal}, the function $\psi$ is then unique up to an additive constant almost everywhere. 

Let us first do a formal linearization.
Assume that the solution to the Monge-Amp\`ere type equation \eqref{MAeq} is of the form
\[
\psi=\varepsilon\varphi+o(\varepsilon), 
\]
and then 
\[
\Phi(x)=\exp_x(\varepsilon\nabla\varphi)+o(\varepsilon).
\]
We can also formally write
\[
\varphi=\frac{\partial\psi}{\partial\varepsilon}\Big\vert_{\varepsilon=0}.
\]
Now
\[
\begin{split}
\frac{f(x)}{h(\Phi(x))}&=\frac{h(x)+\varepsilon f_1(x)}{h(x)+\varepsilon\langle\nabla h(x),\nabla\varphi(x)\rangle_g+o(\varepsilon)}\\
&=(h(x)+\varepsilon f_1(x))\frac{1}{h(x)}\left((1-\varepsilon\frac{1}{h(x)}\langle\nabla h(x),\nabla\varphi(x)\rangle_g+o(\varepsilon)\right)\\
&=1+\varepsilon\frac{f_1(x)}{h(x)}-\varepsilon\frac{\langle\nabla h(x),\nabla\varphi(x)\rangle_g}{h(x)}+o(\varepsilon).
\end{split}
\]

%Take $h\equiv \frac{1}{\mathrm{vol}_g(M)}$ and $f= \frac{1}{\mathrm{vol}_g(M)}(1-\varepsilon\frac{\int_Mf_1\mathrm{d}V_g}{\mathrm{vol}_g(M)})+\frac{\varepsilon f_1}{\mathrm{vol}_g(M)}$, and assume that $\psi$ takes the asymptotic expansian
%\[
%\psi=\varepsilon\varphi+o(\varepsilon).
%\]
%Next, we derive an equation for $\varepsilon$.

Take $y=\Phi(x)$ in the above discussion. Consider the asymptotics
\[
R_i^j(t)=\mathcal{O}(\varepsilon^2),
\]
since $|\dot{\gamma}(t)|=|\dot{\gamma}(0)|=|\nabla\psi(x)|=\mathcal{O}(\varepsilon)$,
and
\[
J_0^1(t)=I_n+\varepsilon  L_0^1(t)+o(\varepsilon).
\]
Substituting into the equation \eqref{JacobiEq}, we have
\[
\ddot{L}_0^1(1)(t)=0,
\]
which is to say
\[
J_0^1(1)=I_n+o(\varepsilon).
\]
Similarly, one can show that
\[
J_1^0(1)=I_n+o(\varepsilon).
\]
Then we can conclude that
\[
\nabla^2_xc(x,\Phi(x))=[J_1^0(1)]^{-1}J^1_0(1)=I_n+o(\varepsilon),
\]
\[
\nabla_x\nabla_yc(x,\Phi(x))=-[J_1^0(1)]^{-1}=-I_n+o(\varepsilon)
\]
Consequently
\[
\det\left[\nabla^2\psi(x)+\nabla^2_xc(x,\Phi(x))\right]=1+\varepsilon\mathrm{trace}(\nabla^2\varphi(x))+o(\varepsilon)=1+\varepsilon\Delta_g\varphi+o(\varepsilon).
\]

Substituting above asymptotic expansions into \eqref{MAeq}, we obtain
\[
1+\varepsilon\Delta_g\varphi+o(\varepsilon)=1+\varepsilon\frac{f_1(x)}{h(x)}-\varepsilon\frac{\langle\nabla h(x),\nabla\varphi(x)\rangle_g}{h(x)}+o(\varepsilon).
\]
Then we obtain the (linear) equation for $\varphi=\frac{\partial\psi}{\partial\varepsilon}\vert_{\varepsilon=0}$,
\begin{equation}\label{linearEq}
\Delta_g\varphi+\langle\nabla\log h,\nabla\varphi\rangle_g=\frac{f_1(x)}{h(x)}.
\end{equation}
Keep in mind that $h$ is not known, only $h\mathrm{d}V_g$ is known, but $\frac{f_1}{h}$ is known. Observe that
\[
h\Delta_g\varphi+\langle\nabla h,\nabla\varphi\rangle_g=f_1\quad\quad\text{in }M.
\]
Integrating the above equation on $M$, and using integration by parts, we have
\[
\int_Mh\Delta_g\varphi+\langle\nabla h,\nabla\varphi\rangle_g\mathrm{d}V_g=\int_M-\langle\nabla h,\nabla\varphi\rangle_g+\langle\nabla h,\nabla\varphi\rangle_g\mathrm{d}V_g=0=\int_M f_1\mathrm{d}V_g,
\]
which is consistent with the condition $\int_Mf_1\mathrm{d}V_g=0$.

Consider the homogeneous equation
\begin{equation}\label{homogeneouseq}
\Delta_g\varphi+\langle\nabla\log h(x),\nabla\varphi\rangle_g=0\quad\text{in }M.
\end{equation}
Using integration by parts, we have
\[
\begin{split}
0=&\int_M (\Delta_g\varphi(x)+\langle\nabla\log h(x),\nabla\varphi(x)\rangle_g)h(x)\varphi(x)\mathrm{d}V_g\\
=&\int_M -\langle\nabla\varphi,\nabla(h\varphi)\rangle_g+\langle\nabla h,\nabla\varphi\rangle_g \varphi\mathrm{d}V_g\\
=&-\int_M h|\nabla\varphi|_g^2\mathrm{d}V_g.
\end{split}
\]
Therefore $\varphi$ is a constant. This means that the solution to the equation \eqref{homogeneouseq} must be a constant. Equivalently, the null space of the operator $P=\Delta_g+\langle\nabla\log h(x),\nabla\rangle_g$ is one-dimensional, consisting of constant functions. We remark here that $P$ is self-adjoint as an operator on the space $L^2(M,\mathrm{d}\nu(x))$, that is, $L^2(M,h(x)\mathrm{d}V_g)$. 

Notice that
\[
\Phi(x)=\exp_x(\varepsilon\nabla\varphi+o(\varepsilon))=x+\varepsilon(\nabla\varphi)^\sharp+o(\varepsilon).
\]
Recall that we have used the convention that $\nabla\varphi=\partial_j\varphi\mathrm{d}x^j$ is a covector, and $(\nabla\varphi)^\sharp$ is the corresponding vector w.r.t. $g$.
Therefore
\[
\frac{\partial \Phi}{\partial\varepsilon}\Big\vert_{\varepsilon=0}=g^{ij}\partial_j\varphi.
\]

To make the above formal linearization rigorous, we can use the following local well-posedness result.
\begin{proposition}
    Give $\alpha\in (0,1)$ and $k\geq 2$. Let $f,h\in C^{k,\alpha}(M)$, $h>0$, there exists a sufficiently small $\varepsilon$ such that if $\|f-h\|_{C^{k,\alpha}}<\varepsilon$, then $\psi\in C^{k+2,\alpha}(M)$.
\end{proposition}
\begin{proof}
    We follow the proof of \cite[Proposition 2.1]{LiLin2025}. Fix $h$. We set up the Banach spaces
    \[
    X:=C^{k+2,\alpha}(M)/+,\quad Y:=C^{k,\alpha}(M).
    \]
    Consider the map $F:X\rightarrow Y$,
    \[
     F(u)=\frac{h(\exp_x(\nabla u(x)))\det\left[\nabla^2u(x)+\nabla^2_xc(x,\exp_x(\nabla u(x)))\right]}{|\det(\nabla_y\nabla_xc(x,\exp_x(\nabla u(x)))|}
    \]
    Clearly $F(0)=h$ and $F$ is differentiable at near $0$. By above analysis, we have
    \[
    \partial_u F(0)v=h\Delta_gv+\langle\nabla h,\nabla v\rangle_g.
    \]
    For any $f_1\in C^{k,\alpha}(M)$, the equation
    \[
    \partial_u F(0)v=h\Delta_gv+\langle\nabla h,\nabla v\rangle_g=f_1
    \]
    has a unique solution $v\in C^{k+2,\alpha}(M)/+$ using the standard Schauder theory. We can then apply \cite[Theorem 10.4]{renardy2004introduction} to conclude that for any $f$ in a sufficiently small neighborhood $V$ of $F(0)=h$ (in $C^{k,\alpha}$-norm), there exists a solution $\psi$ to \eqref{MAeq}. The uniqueness of $\psi$ is already guaranteed by the nonlinear theory. In addition, \cite[Theorem 10.4]{renardy2004introduction} says that the map $F^{-1}:V\rightarrow X, F(f)=\psi$ is continuously differentiable.
\end{proof}

Therefore we can have $g^{ij}\partial_j\varphi$ from $\Phi$.
Now our problem reduces to the inverse problem for the linear equation \eqref{linearEq}: fix $\nu=h\mathrm{d}V_g$, assume for any given $\frac{f_1}{h}$ such that $\int_M \frac{f_1}{h}\mathrm{d}\nu=0$, we can measure $g^{ij}\partial_j\varphi(x)$, $x\in M$, where $\varphi$ is a solution to \eqref{linearEq}, we want to recover $g$.\\

\section{Proof of the main result}
Recovering the coefficients of elliptic equations from internal measurements also arise from hybrid medical imaging. We refer to \cite{bal2013reconstruction,monard2013inverse,bal2014linearized,bal2014inverse} for interested readers. In particular, the problem considered in \cite{bal2014inverse} is quite similar to ours for the linearized equation \eqref{linearEq}.\\

First, we consider the local recovery of $g$.
Assume $X\subset M$ is a simply connected convex subset of $M$. For any $u\in C^\infty(X)$ such that $Pu=0$ in $X$, we can extend $u$ smoothly to $M$ such that $Pu=f'$ with some $f'\in C^\infty(M)$, $f'=0$ in $X$. Let $E=\{u\in H^1(X),Pu=0\}$, $F=\{u\vert_X,u\in H^1(M),Pu=\frac{f_1}{h},\frac{f_1}{h}\in L^2(M),f_1\equiv 0\text{ on }X,\int_Mf_1\mathrm{d}V_g=0\}$. Note that $F$ is a linear subspace of $E$. The discussion in previous section shows that $\overline{F}=E$ with the strong $L^2(X,\mathrm{d}\nu)$ topology.\\

We consider local recovery of $g$ in $X$ using solutions to the equation $Pu=0$ in $X$, which in local coordinates can be written as
\[
\partial_i(|g|^{1/2}g^{ij}h\partial_ju)=0.
\]
Note that we can use all solutions of $Pu=0$ on $X$ by discussions in previous paragraph.
Denote the matrix $\gamma$ as $\gamma^{ij}=|g|^{1/2}g^{ij}h$, we can write the above equation as
\[
\partial_i(\gamma^{ij}\partial_ju)=0,
\]
which is the precisely equation considered in \cite{bal2014inverse}. However, what we can measure is $g^{ij}\partial_j u$, not $\gamma^{ij}\partial_j u$. Nevertheless, the recovery of $g$ follows directly from the arguments in \cite{bal2014inverse} and in the following we briefly summarize the main steps. If $X$ is small enough, we have the following propositions. For proofs of these propositions, see \cite{bal2014inverse}.
\begin{proposition}\label{hyph1}
There exist two solutions $(u_1,u_2)$ of $Pu=0$ in $X$ satisfying
\[
\inf_{x\in X}\mathcal{F}_1(u_1,u_2)\geq c_0>0,\text{ where } \mathcal{F}_1(u_1,u_2):=|\nabla u_1|^2|\nabla u_2|^2-\langle \nabla u_1,\nabla u_2\rangle^2.
\]
\end{proposition}
\begin{proposition}\label{hyph2}
There exist n solutions $(u_1,u_2,\cdots, u_n)$ of $Pu=0$ in $X$ satisfying
\[
\inf_{x\in X}\mathcal{F}_2(u_1,u_2,\cdots,u_n)\geq c_0>0,\text{ where } \mathcal{F}_1(u_1,u_2,\cdots,u_n):=\det(\nabla u_1,\nabla u_2,\cdots,\nabla u_n).
\]
\end{proposition}

Choose $u_1,u_2,\cdots,u_n$ be solutions satisfying Proposition \ref{hyph2} and consider additional solutions $u_{n+1},\cdots,u_{n+m}$. Each additional solutions can be decomposed into
\[
\nabla u_{n+k}=\sum_{i=1}^n\mu^i_k\nabla u_i,\quad 1\leq k\leq m.
\]
Denote $H_i=g^{-1}\nabla u_i$, and note that $H_i$ is what we can measure. It was shown in \cite{bal2014linearized} that
\[
\begin{split}
\mu^i_k=&-\frac{\det(\nabla u_1,\cdots,\overbrace{\nabla u_{n+k}}^{i},\cdots,\nabla u_n)}{\det(\nabla u_1,\cdots,\nabla u_n)}\\
=&-\frac{\det(\gamma\nabla u_1,\cdots,\overbrace{\gamma\nabla u_{n+k}}^{i},\cdots,\gamma\nabla u_n)}{\det(\gamma\nabla u_1,\cdots,\gamma\nabla u_n)}\\
=&-\frac{\det(H_1,\cdots,\overbrace{H_{n+k}}^{i},\cdots,H_n)}{\det(H_1,\cdots,H_n)}.
\end{split}
\]
We introduce the matrices $Z_k$, for $k=1,2,\cdots,m$, defined by
\[
Z_k=[Z_{k,1}\vert\cdots\vert Z_{k,n}], \text{ where } Z_{k,i}:=\nabla\mu^i_k.
\]

Here and below, $A_n(\mathbb{R})$ and  $S_n(\mathbb{R})$ denote the vector spaces of anti-symmetric and symmetric $n\times n$ real matrices. Denote also $A^{\mathrm{sym}}:=\frac{1}{2}(A+A^T)$.
\begin{proposition}\label{hyph3}
Assume that $(u_1,u_2,\cdots, u_n)$ are solutions over $X$ satisfying the conditions in Proposition \ref{hyph2} and denote by $H$ the matrix with columns $H_1,H_2,\cdots, H_n$. Then there exists $u_{n+1},\cdots,u_{n+m}$ solutions of $Pu=0$ on $X$ such that
\begin{equation}\label{constraintforg}
\mathcal{W}:=\mathrm{span}\{(Z_kH^T\Omega)^{\mathrm{sym}},\Omega\in A_n(\mathbb{R}),1\leq k\leq m\}\subset S_n(\mathbb{R})
\end{equation}
has codimension one in $S_n(\mathbb{R})$ thoughout $X$.
\end{proposition}

As discussed in \cite[Section 4]{bal2014inverse}, Propositions \ref{hyph1}, \ref{hyph2} and \ref{hyph3} hold when $\gamma$ is constant, for which explicit construction of $\{u_j\}_{j=1}^{n+m}$ is possible. By continuity, the above propositions also hold true for variable $\gamma$ if $X$ is small enough.

It was proved in \cite{bal2014inverse} such that from the constraint \eqref{constraintforg}, one can recover $\tilde{g}=(\det g)^{1/n}g^{-1}$ pointwisely in $X$. Note that $\det\tilde{g}=1$. Actually $\tilde{g}$ is in the orthogonal complement of $\mathcal{W}$ in \eqref{constraintforg}.
%We write down the equation
%\[
%\frac{1}{\det(g)}\frac{\partial}{\partial x^i}(g^{ij}h\sqrt{\det g}\partial_ju)=0.
%\]
Then denote the factor $\beta=(\det g)^{-1/n}$, and use two solutions $u_1, u_2$ satisfying the condition in Proposition \ref{hyph1}.
We write
\[
\nabla u_j=\mathrm{d}u_j=\beta^{-1}\tilde{g}^{-1}H_j.
\]
It turns out that $\beta$ satisfies the gradient equation (cf. \cite[(12)]{bal2014inverse})
\begin{equation}\label{eqforbeta}
\begin{split}
\nabla\log\beta=&\frac{1}{D|H_1|^2}\left(|H_1|^2\mathrm{d}(\tilde{g}^{-1}H_1)-(H_1\cdot H_2)\mathrm{d}(\tilde{g}^{-1}H_2)\right)(\tilde{g}H_1,\tilde{g}H_2)\tilde{g}^{-1}H_1\\
&-\frac{1}{|H_1|^2}\mathrm{d}(\tilde{g}^{-1}H_1)(\tilde{g}H_1,\cdot),\quad \text{in }X,
\end{split}
\end{equation}
where $D:=|H_1|^2|H_2|^2-(H_1\cdot H_2)^2$ is bounded away from $0$ in $X$.
Since $\tilde{g}$ is already recovered, this allows us to reconstruct $\beta$ in $X$ under the knowledge of $\beta(x_0)$ at a fixed point $x_0\in X$. For global reconstruction, one can cover the whole manifold $M$ with a finite number of small open domains $X_1,X_1,\cdots, X_N$ and patch together the local reconstructions using a partition of unity. Since the value of $\beta$ at one point $x_0$ can be arbitrarily chosen, we can recover $\log\beta$ up to an additive constant by solving the first order equation \eqref{eqforbeta}. This in turn proves that one can recover $\beta$, and thus $g$, up to a multiplicative constant. This proves the main result.

%\begin{remark}
%In the above argument we can actually fix $\nu$ and use only different $\mu$'s.
%\end{remark}
%\begin{theorem} 
% and $X'\subset\subset X$. Let $u_0$ be a solution of $Pu_0=0$ on $X$. Then there exists a function $f$ on $M$, such that the solution of $Lu_\varepsilon=f$ on $M$ satisfies
%\[
%\|u_\varepsilon-u_0\|_{C^{2,\alpha}(X')}\leq\varepsilon.
%\]
%\end{theorem}
%\begin{proof}
%Let $E=\{u\in H^1(X),Pu=0\}$, $F=\{u\vert_X,u\in H^1(M),Pu=f,f\in L^2(M),f\equiv 0\text{ on }X\}$. Note that $F$ is a linear subspace of $E$. We will prove $\overline{F}=E$ with the strong $L^2(X,\mathrm{d}\nu)$ topology. By Hahn Banach, this means that for any $\phi\in L^2(X)$, then $\langle \phi,u\rangle=0$ for all $u\in F$ implies that $\langle\phi,u\rangle=0$ for all $u\in E$.
%
%Take $\phi\in L^2(X)$ such that $\langle\phi,u\rangle_{L^2(X)}=0$ for any $u\in F$. Extend $\phi$ by $0$ outside $X$ and still denote the function by $\phi$. Let $v$ be the solution of
%\[
%Pv=\phi\quad\text{in }M.
%\]
%Integration by parts yields
%\[
%\begin{split}
%0=\langle Pu,v\rangle_{L^2(M)}-\langle u,Pv\rangle_{L^2(M)}=\langle Pu,v\rangle_{L^2(M\setminus X)}=\langle f,v\rangle_{L^2(M\setminus X)}.
%\end{split}
%\]
%Then $v\equiv 0$ on $M\setminus X$. For any $
%\end{proof}
\bibliographystyle{abbrv}
\bibliography{biblio.bib}
\end{document}